\documentclass{amsart}

\usepackage{amsmath, amssymb, amsthm, natbib, graphicx, url}

\usepackage{thmtools}

\usepackage[noend,ruled,algo2e,linesnumbered]{algorithm2e}

\usepackage{hyperref}
\usepackage{cleveref}

\usepackage{tikz}
\usetikzlibrary{calc,intersections,arrows.meta}

\title{Neural Networks and (Virtual) Extended Formulations}
\usepackage{times}

\author{Christoph Hertrich}
\address{Christoph Hertrich\\University of Technology Nuremberg}
\email{christoph.hertrich@utn.de}

\author{Georg Loho}
\address{Georg Loho\\Freie Universität Berlin \& University of Twente}
\email{georg.loho@math.fu-berlin.de}

\newtheorem{theorem}{Theorem}[section]
\newtheorem{proposition}[theorem]{Proposition}
\newtheorem{lemma}[theorem]{Lemma}

\newtheorem{question}[theorem]{Question}

\theoremstyle{definition}

\theoremstyle{remark}
\newtheorem{remark}[theorem]{Remark}

\usepackage{enumerate}
\usepackage{mathtools}
\usepackage{tikz}

\usepackage{todonotes}
\newenvironment{proofwithcaption}[1]{%
	\renewcommand{\proofname}{#1}%
	\begin{proof}%
	}{%
	\end{proof}%
}

\newcommand{\R}{\mathbb{R}}

\newcommand{\N}{\mathbb{N}}

\newcommand{\dinv}{\delta^\mathrm{in}_v}

\DeclareMathOperator{\conv}{conv}
\DeclareMathOperator{\xc}{xc}
\DeclareMathOperator{\vxc}{vxc}
\DeclareMathOperator{\nnc}{nnc}
\DeclareMathOperator{\mnnc}{mnnc}
\DeclareMathOperator{\epi}{epi}

\begin{document}

\begin{abstract}	
	Neural networks with piecewise linear activation functions, such as rectified linear units (ReLU) or maxout, are among the most fundamental models in modern machine learning.
	We make a step towards proving lower bounds on the size of such neural networks by linking their representative capabilities to the notion of the extension complexity $\xc(P)$ of a polytope $P$. This is a well-studied quantity in combinatorial optimization and polyhedral geometry describing the number of inequalities needed to model $P$ as a linear program. We show that $\xc(P)$ is a lower bound on the size of any \emph{monotone} or \emph{input-convex} neural network that solves the linear optimization problem over $P$. This implies exponential lower bounds on such neural networks for a variety of problems, including the polynomially solvable maximum weight matching problem.
	
	In an attempt to prove similar bounds also for general neural networks, we introduce the notion of \emph{virtual extension complexity} $\vxc(P)$, which generalizes $\xc(P)$ and describes the number of inequalities needed to represent the linear optimization problem over $P$ as a difference of two linear programs.
	We prove that $\vxc(P)$ is a lower bound on the size of any neural network that optimizes over $P$.
	While it remains an open question to derive useful lower bounds on $\vxc(P)$, we argue that this quantity deserves to be studied independently from neural networks by proving that one can efficiently optimize over a polytope $P$ given a virtual extended formulation with small encoding size.
\end{abstract}

\maketitle
\section{Introduction}

A \emph{feedforward neural network} is a directed, acyclic graph in which each vertex (\emph{neuron}) defines a simple computation, usually a linear transformation composed with a scalar-valued, \emph{continuous and piecewise linear} (CPWL) activation function.
While a standard choice for the activation function is the \emph{rectified linear unit} (ReLU) $x\mapsto\max\{0,x\}$, in this paper we focus on the more general \emph{maxout} networks.
These allow to compute the maximum of constantly many linear functions at each neuron.
As a result, the entire network computes a (potentially complex) CPWL function.
One of the big challenges in the theoretical analysis of neural networks is to understand how many neurons one requires to exactly or approximately represent a given (CPWL) function.
To the best of our knowledge, it is an open question whether there exists a family of CPWL functions, which we can evaluate in polynomial time, but which cannot be represented by polynomial-size neural networks.\footnote{Note that here we refer to exact real-valued and not binary computation, compare the related discussion in \Cref{sec:relatedwork}.}

The piecewise linear nature of the studied networks suggests to tackle such questions by means of polyhedral geometry, see, e.g., the recent survey by \citet{huchette2023deep}.
In fact, a similar problem to the question above used to be open for a long time in the context of linear programming, until \citet{rothvoss2017matching} resolved it affirmatively: does there exist a polytope $P$ over which we can optimize in polynomial time, but any linear programming formulation must have exponential size?
This question can be formalized with the notion of \emph{extension complexity} $\xc(P)$, which describes the minimal number of facets of any polytope $Q$ that projects onto $P$.
In this case, we call $Q$ an \emph{extended formulation} of $P$. \citet{rothvoss2017matching} proved that the matching polytope has exponential extension complexity even though the algorithm by \citet{edmonds1965maximum} can be used to find the maximum weight matching or minimum weight perfect matching in polynomial time.
Additionally, \citet{fiorini2015exponential} proved that a couple of polytopes associated with NP-hard optimization problems like the traveling salesperson problem have exponential extension complexity. \citet{rothvoss2017matching} and \citet*{fiorini2015exponential} received the Gödel Prize 2023 for their breakthrough results.

There is a direct translation between polytopes (as feasible sets of linear programs) and CPWL functions (represented by neural networks) through the notion of the \emph{support function} $f_P(c)=\max_{x\in P} c^\top x$ of a polytope $P$, see \Cref{fig:correspondence-polytope-support-function}.
This (convex) CPWL function, which has one linear region for each vertex of $P$, uniquely determines $P$ via convex duality.
Computing $f_P(c)$ means determining the objective value when optimizing over $P$ in $c$-direction.
Each CPWL function $f_P$ can be represented by a neural network \citep{arora2018understanding}, though the required number of neurons can be large. 
To quantify this, we define the \emph{neural network complexity} $\nnc(P)$ as the minimum number of neurons to represent $f_P$ by a maxout neural network.

\begin{figure}[ht]
  \centering
\begin{tikzpicture}[scale=0.65]

          \footnotesize\textcolor{blue!70!black}{
            \node[draw, color=blue, circle, fill=blue, inner sep=1, label=270:{(0,0)}] (null) at (-1,-1) {};
            \node[draw, color=blue, circle, fill=blue, inner sep=1, label=90:{(0,1)}] (zwei) at (-1,1) {};
            \node[draw, color=blue, circle, fill=blue, inner sep=1, label=270:{(1,0)}] (eins) at (1,-1) {};
            \draw[thick, color=blue] (null) -- (eins) -- (zwei) -- (null);
        }

          \begin{scope}[shift={(8.7,0)}]
        \draw[->, bend right=20] (-2.7,.7) to node[above] {\footnotesize Newton Polytope} (-6.5,.7);
        \draw[->, bend right=20] (-6.5,-.7) to node[below] {\footnotesize Support Function} (-2.7,-.7);
        \node at (-4.6,0) {\footnotesize \textbf{Duality}};
        \end{scope}
        
  \begin{scope}[shift={(8.4,.4)}]
            \textcolor{blue!70!black}{
                \draw[] (0,0) -- (1,1);
                \draw[] (0,-1.2) -- (0,0);
                \draw[] (-1.2,0) -- (0,0);
                \footnotesize
                \node at (0.5,-0.2) {$x_1$};
                \node at (-0.5,-0.5) {$0$};
                \node[align=right] at (-0.2,0.5) {$x_2$};
                \node at (0,-1.8) {$\max\{0,x_1,x_2\}$};}
        \end{scope}
        
    \end{tikzpicture} 
  \caption{Sketch of correspondence between polytopes and support functions. }
  \label{fig:correspondence-polytope-support-function}
\end{figure}

\subsection{Our Contributions}

The aim of this paper is to connect the world of extended formulations with the study of neural networks.
A ``dream result'' in this direction would be to bound $\xc(P)$ polynomially in $\nnc(P)$.
Then the breakthrough results on extension complexity would directly imply strong lower bounds on the size of neural networks.
It turns out, however, that there is one feature of neural networks that seems to make the ``dream result'' difficult to obtain or maybe even wrong: namely the ability to use subtraction. There are two natural ways to circumvent this difficulty: either prove the ``dream result'' for a weaker version of neural networks, or prove it for a stronger version of extension complexity. We accomplish both of these variants in our paper.

\subsubsection*{Monotone and Input-Convex Neural Networks}
A consequence of the discussion above is that, if we remove the ability to subtract within neural networks, we do indeed obtain our ``dream result'', namely lower bounds through $\xc(P)$. This leads to two different, but closely related neural network models: \emph{monotone} \citep{mikulincer2022size} and \emph{input-convex} neural networks (ICNNs; \citet{amos2017input}). The former only allows nonnegative weights, therefore enforcing each neuron and the entire network to represent a monotone and convex function. The latter allows negative weights only on outgoing connections from the input neurons, therefore also enforcing the represented function to be convex, but no longer monotone.

Obviously, every network that is monotone in the sense defined above is also input-convex. In \Cref{sec:monotone}, we show that for computing monotone functions with maxout activations, ICNNs are not more efficient than monotone networks. So, in this case, both models are equivalent. In analogy to $\nnc(P)$, we define the \emph{monotone neural network complexity} $\mnnc(P)$ as the minimum number of neurons to represent $f_P$ by a monotone network or an ICNN.\footnote{Even though ICNNs are used more frequently in practice, we still prefer to name the complexity measure after monotone networks due to the significance of monotone models in the circuit complexity community.} We then show that exponential lower bounds on $\xc(P)$ imply corresponding lower bounds on $\mnnc(P)$ and consequently on exact and approximate representations with monotone networks and ICNNs.

Studying these restricted types of neural networks is justified both from a theoretical and a practical perspective.
From the theoretical perspective, it is a natural approach in complexity theory to prove lower bounds first for monotone models of computation, to circumvent some additional challenges of the general case; see, e.g., \citet{valiant1979negation}.
From the practical perspective, it is sometimes desirable or even necessary to incorporate prior knowledge about monotonicity and/or convexity of the target function into a machine learning model. For convexity, ICNNs have been used extensively for exactly this purpose after they were introduced by \citet{amos2017input}; see, e.g., the literature overview in \citet{gagneux2025convexity}. For monotonicity, we refer to \citet{mikulincer2022size} and the references therein for recent studies of monotone neural networks in the machine learning community.

Depending on whether similar lower bounds can also be established for general neural networks or not, our exponential lower bounds on monotone networks and ICNNs could indicate that these models might not be the best solution for enforcing convexity of a machine learning model. This aligns with some observations by \citet{gagneux2025convexity}.

\subsubsection*{Virtual Extension Complexity.}

In order to pave the way towards lower bounds for general neural networks, we propose the notion of \emph{virtual extension complexity}
\[
\vxc(P)=\min\{\xc(Q)+\xc(R)\mid \text{$Q$ and $R$ are polytopes with } P+Q=R\}\ ,
\]
where $P+Q=\{p+q\mid p\in P, q\in Q\}$ is the \emph{Minkowski sum}, see \Cref{fig:Minkowski-sum}.
In this definition,~$P$ is a (formal) \emph{Minkowski difference} of two polytopes~$R$ and~$Q$.
Note that this is not the same as $R + (-1)\cdot Q$ but rather the inverse operation of Minkowski addition.
While we start with a convex polytope $P$ in the first place, this notion of Minkowski difference does in general not result in a polytope.
More precisely, it yields a polytope if and only if the normal fan of $R$ refines the normal fan of $Q$, see~\citet{PostnikovReinerWilliams:2008}.
The name \emph{virtual} extension complexity is derived from \emph{virtual polytopes} \citep{panina2015virtual}, a framework for the algebraic study of formal Minkowski differences of polytopes.
Observe that $\vxc(P)\leq \xc(P)$ because we can always choose $Q=\{0\}$ with $\xc(\{0\})=0$.
In that sense, virtual extension complexity is really a strengthening of the ordinary extension complexity.

\begin{figure}[htb]
	\centering
	\begin{tikzpicture}[scale=.8, line join=round, line cap=round]
		
		\tikzset{
			polyA/.style={draw=blue!70!black, fill=blue!25, thick},
			polyB/.style={draw=red!70!black, fill=red!20, thick},
			sumP/.style={draw=green!50!black, fill=green!25, thick},
			copy/.style={draw=gray!60, fill=gray!15, dashed, thin},
		}
		
		\coordinate (A1) at (0,0);
		\coordinate (A2) at (2,0);
		\coordinate (A3) at (2.8,1.2);
		\coordinate (A4) at (1.2,2.2);
		\fill[polyA] (A1) -- (A2) -- (A3) -- (A4) -- cycle;
		\node[blue!70!black] at (1.5,0.9) {$A$};
		
		\begin{scope}[shift={(.4,0)}]
			\coordinate (B1) at (3,0);
			\coordinate (B2) at (4.4,0.1);
			\coordinate (B3) at (4.0,1.4);
			\fill[polyB] (B1) -- (B2) -- (B3) -- cycle;
			\node[red!70!black] at (3.8,0.55) {$B$};
		\end{scope}

		\begin{scope}[shift={(.8,0)}]
			\coordinate (S11) at (5,0);
			\coordinate (S12) at (6.4,0.1);
			\coordinate (S13) at (6.0,1.4);
			\coordinate (S21) at (7.0,0.0);
			\coordinate (S22) at (8.4,0.1);
			\coordinate (S23) at (8.0,2.1);
			\coordinate (S31) at (7.8,1.2);
			\coordinate (S32) at (9.2,1.3);
			\coordinate (S33) at (8.8,2.6);
			\coordinate (S41) at (6.2,2.2);
			\coordinate (S42) at (7.6,2.3);
			\coordinate (S43) at (7.2,3.6);
			\fill[sumP] (S11) -- (S21) -- (S22) -- (S32) -- (S33) -- (S43) -- (S41) -- cycle;
			\node[green!50!black] at (7.3,1.5) {$A+B$};
		\end{scope}

		\node[black] at (3.2,0.7) {$+$};
		\node[black] at (5.3,0.7) {$=$};
		
	\end{tikzpicture}
	\caption{Minkowski sum of two polygons. }
	\label{fig:Minkowski-sum}
\end{figure}

In \Cref{sec:nntovirtual}, we deduce that $\vxc(P)$ is indeed a lower bound for $\nnc(P)$, up to a constant factor.
This leaves the open question to find ways to lower-bound $\vxc(P)$ in order to achieve the original goal to lower-bound $\nnc(P)$.
To this end, it seems to be crucial to generally obtain a better understanding of the complexity measure $\vxc(P)$ and its relation to $\xc(P)$. Observe that $P+Q=R$ is equivalent to $f_P=f_R-f_Q$ pointwise.
Therefore, intuitively, in order to optimize over $P$, one only needs to optimize over $R$ and $Q$ and subtract the results.
We make this intuition formal in \Cref{sec:optimize}, implying that extended formulations for $Q$ and $R$ with small encoding size are sufficient to optimize efficiently over $P$.
Furthermore, in \Cref{sec:genperm}, we provide a class of examples with $P+Q=R$ demonstrating that $\xc(R)$ can be much smaller than $\xc(P)$.
This gives important insights on how extension complexity behaves under Minkowski sum and implies that we really need to look at \emph{both} $\xc(Q)$ and $\xc(R)$ in order to lower-bound $\vxc(P)$.

Overall, we now have four different ways to represent a polytope or its support function through (virtual) extended formulations and (monotone) neural networks. \Cref{fig:relation-measures} shows what we know about how the associated complexity measures $\xc(P)$, $\vxc(P)$, $\nnc(P)$, and $\mnnc(P)$ relate to each other.

\begin{figure}[h]
	\def\dist{-2.4}
	\def\rowgap{-1.6}
	\centering
	\begin{tikzpicture}
		\begin{scope}[every node/.style= {}] 
			\node (xc) at (\dist,-.5*\rowgap) {$\xc(P)$};
			\node (mnnc) at (0,0) {$\mnnc(P)$};
			\node (vxc) at (2*\dist,0) {$\vxc(P)$};
			\node (nnc) at (\dist,.5*\rowgap) {$\nnc(P)$};
		\end{scope}
		
		\begin{scope}[every node/.style={fill=white,circle,font=\footnotesize,inner sep=2pt},
			every edge/.style={draw,black,thick}]
			\path [->] (vxc) edge (xc);
			\path [->] (xc) edge (mnnc);
			\path [->] (vxc) edge (nnc);
			\path [->] (nnc) edge (mnnc);
			\path [->] (xc) edge[out=240, in=120] node{\bfseries ?} (nnc);
			\path [->] (nnc) edge[out=60, in=300] node{\bfseries ?} (xc);
			
		\end{scope}
	\end{tikzpicture}
	\caption{Relations between complexity measures for a polytope $P$. A directed arc means that the tail can be polynomially bounded by the head. It remains an open question whether $\xc(P)$ and $\nnc(P)$ can be related this way.}
	\label{fig:relation-measures}
\end{figure}

\subsection{Further Related Work}\label{sec:relatedwork}

In this paper our lever to prove lower bounds on neural networks is combinatorial optimization problems. Complementing upper bounds were established by \citet{hertrich2024relu} for minimum spanning trees and maximum flows and by \citet{hertrich2025arithmetic} for regular matroids. The latter result also shows that $\vxc(P)\in O(n^3)$ for $P$ being the base polytope of a regular matroid with $n$ elements. This is a first example for which $\vxc(P)$ is smaller than the best known upper bound on $\xc(P)$, which is $O(n^6)$ in this case~\citep{aprile2022regular}.
Furthermore, \citet{hertrich2023provably} proved neural network constructions for the knapsack problem, even though they are different in flavor because some integrality assumptions are made.

Concerning the general expressivity of (piecewise linear) neural networks, the celebrated \emph{universal approximation theorems} state that a single layer of neurons is sufficient to approximate any continuous function on a bounded domain; see \citet{cybenko1989approximation} for the original version for sigmoid activation functions and \citet{leshno1993multilayer} for a version that encompasses ReLU. However, such shallow neural networks usually require a large number of neurons. A sequence of results demonstrates that deeper networks sometimes require exponentially fewer neurons to represent the same functions; see, e.g., \citet{arora2018understanding, eldan2016power, telgarsky2016benefits}. While these works contain exponential lower bounds on the size of neural networks, they are focused on shallow networks. In contrast, we aim to prove lower bounds regardless of the depth.

In terms of exact representation, it is known that a function can be represented if and only if it is CPWL \citep{arora2018understanding}, and it is still an open question whether constant depth is sufficient to do so \citep{hertrich2023towards,haase2023lower,grillo2025depth,bakaev2025better}. Interestingly, also for this question, monotone networks and ICNNs seem to be more amenable for proving lower bounds than their non-monotone counterparts \citep{valerdi2024minimal,bakaev2025depth}. Furthermore, the related question of how to efficiently write a non-convex CPWL function as a difference of two convex ones received quite some attention recently \citep{brandenburg2024decomposition, tran2024minimal}. A neural network architecture that explicitly learns differences of convex functions was proposed by \citet{sankaranarayanan2022cdinn}.

We would like to emphasize that we view neural networks as a model of \emph{real-valued} computation, as opposed to binary models of computation like Boolean circuits and Turing machines. In fact, if one restricts the inputs of a neural network to be binary, it is not too difficult to simulate AND-, OR-, and NOT-gates \citep{mukherjee2017lower}. Thus, in such a binary model, every problem in P can be solved with polynomial-size neural networks. However, such networks would usually be very sensitive to single bits in the input. This is undesirable for practical neural networks and makes it impossible to transform these constructions naturally into exact or approximate neural networks in the real-valued model, compare the discussion by \citet{hertrich2024relu}.
The more useful connection to circuit complexity is through arithmetic \citep{shpilka2010arithmetic}, and in particular tropical circuits \citep{jukna2023tropical}, which are also real-valued models of computation. For example, \citet{jukna2020reciprocal} study the effect of subtraction at different locations within tropical circuits, which is closely related to the idea of differences of monotone neural networks and ICNNs. Again, for a more detailed discussion on the relation of neural networks to arithmetic and tropical circuits, we refer to \citet{hertrich2024relu}.

In fact, the extension complexity has been related before to Boolean and arithmetic circuits, see \citet{fiorini2021strengthening,hrubevs2023shadows}. This is also related to the proof that the permutahedron has extension complexity $\mathcal{O}(n\log n)$ \citep{Goemans:2015}, as this goes via sorting networks, which can be seen as a very specific version of a piecewise-linear arithmetic circuit.

\section{Preliminaries}\label{sec:prelim}

\subsection*{Polyhedra and Polytopes.}

A \emph{polyhedron} $P$ is a finite intersection of halfspaces $P=\{x\in\R^d\mid Ax\leq b\}$; the representation as such an intersection is called $H$-representation.
The \emph{affine hull} of $P$ is the smallest affine subspace containing~$P$ and the \emph{dimension} $\dim(P)$ is defined as the dimension of its affine hull.
A \emph{face} $F$ of~$P$ is the set of maximizers over $P$ with respect to a linear objective function: $F=\arg\max \{c^\top x\mid x\in P\}$ for some $c \in \mathbb{R}^d$. 
Faces of polyhedra are polyhedra themselves.
Zero-dimensional faces are called \emph{vertices}.
Faces of dimension $\dim(P) - 1$ are called \emph{facets}; let $h(P)$ be the number of such facets of $P$.
The minimal $H$-representation contains precisely one inequality for each facet and potentially a bunch of equalities to describe the affine hull of $P$.

If $P$ is bounded, it is also called a \emph{polytope}. By the Minkowski-Weyl theorem, a polytope $P$ can equivalently be written as convex hull of finitely many points. The inclusion-wise minimal such representation is precisely the convex hull of the vertices $V(P)$, the \emph{$V$-representation}.
We set $v(P) = \lvert V(P)\rvert$.

Each of $V$- and $H$-representation of a polytope can be exponentially smaller than the other one.
For example, the $d$-dimensional cube $\{x\in \R^d\mid \lVert x \rVert_\infty \leq 1\}$ has $2^d$ vertices but only $2d$ facets.
On the other hand, the $d$-dimensional cross-polytope $\{x\in \R^d\mid \lVert x \rVert_1 \leq 1\}$ has only $2d$ vertices, but $2^d$ facets. Standard references for polytope theory are \citet{grunbaum2003convex,ziegler2012lectures}.

\subsection*{Extended Formulations.}
Sometimes, a more compact way of representing a polyhedron $P$ can be obtained by representing it as the projection of a higher-dimensional polyhedron $Q\subseteq\R^{e}$ with $e\geq d$. Such a polyhedron $Q$ with $\pi(Q)=P$ for an affine projection $\pi$ is called an \emph{extended formulation} of $P\subseteq\R^d$.
The \emph{extension complexity} of $P$ is defined as $\xc(P)=\min\{h(Q)\mid \pi(Q)=P\}$. 
The extension complexity of a polytope $P$ is upper bounded by $\min\{v(P),h(P)\}$, but can be exponentially smaller. For example, the spanning tree polytope of a graph with $n$ vertices has extension complexity $\mathcal{O}(n^3)$ \citep{martin1991using}, but the polytope itself has exponentially many vertices and facets in $n$. See \citet{conforti2013extended} for a (not very recent) survey on extended formulations.

\subsection*{Support Functions.} For a polytope $P\subseteq \mathbb{R}^d$, its \emph{support function} $f_P \colon \mathbb{R}^d \to \mathbb{R}$ is defined as $f_P(c)=\max_{x\in P} c^\top x$, see \Cref{fig:correspondence-polytope-support-function}.
In other words, it maps a linear objective direction to the optimal value obtained by optimizing over the polytope.
Support functions are convex, continuous, piecewise linear, and \emph{positively homogeneous} functions; and every function with these properties is a support function of a polytope. Here, a function $f \colon \mathbb{R}^d \to \mathbb{R}$ is \emph{positively homogeneous} if $f(\lambda x)=\lambda f(x)$ for all scalars $\lambda\geq 0$.
Let $\mathcal{F}^d$ be the set of all such functions from $\R^d\to\R$ and let $\mathcal{P}^d$ be the set of all polytopes embedded in $\R^d$. Then the map $\mathcal{P}^d\to\mathcal{F}^d$, $P\mapsto f_P$, is a bijection that is compatible with a certain set of operations, see \Cref{fig:sum-support-functions}:
\begin{enumerate}[(i)]
	\item $f_{P+Q} = f_P + f_Q$, where the ``+'' on the left-hand side is \emph{Minkowski sum}; 
	\item $f_{\lambda P} = \lambda f_P$ for $\lambda\geq 0$, where $\lambda P = \{\lambda p\mid p\in P\}$ is the \emph{dilation} of $P$ by~$\lambda$; 
	\item $f_{\conv(P\cup Q)} = \max\{f_P, f_Q\}$.
\end{enumerate}

The inverse map that maps a support function $f_P$ to its unique associated polytope $P$ is well-studied in tropical geometry.
In fact, borrowing the name from (tropical) polynomials, we call $P$ the \emph{Newton polytope} of $f_P$.
The connection between CPWL functions and their Newton polytopes has previously been used in the study of neural networks and their expressivity \citep{zhang2018tropical,montufar2022sharp,haase2023lower,hertrich2023towards,brandenburg2024real,misiakos2022neural}.

\begin{figure}[htb]
  \centering

  \begin{tikzpicture}[scale=0.4]

    \begin{scope}[shift={(-8.6,0)}]
      \draw[help lines] (-2.2,-1.2) grid (3.4,4.2);
  \draw[black,->,thick] (-2.2,0) -- (3.4,0);
  \draw[black,->,thick] (0,-1.2) -- (0,4.2);

  \draw[ultra thick,red] (-1,0) -- (2,0);
  \draw[thick,blue] (-2.1,2.1) -- (0,0);
  \draw[thick,blue] (0,0) -- (2.2,4.4);
    \end{scope}

    \node[black] at (-3.7,1.5) {$+$};

  \begin{scope}[shift={(0,0)}]
  \draw[help lines] (-2.2,-1.2) grid (3.4,4.2);
  \draw[black,->,thick] (-2.2,0) -- (3.4,0);
  \draw[black,->,thick] (0,-1.2) -- (0,4.2);

  \draw[ultra thick,red] (-1,0) -- (1,0);
  \draw[thick,blue] (-2.1,2.1) -- (0,0);
  \draw[thick,blue] (0,0) -- (3.2,3.2);
\end{scope}

  \node[black] at (4.7,1.5) {$=$};
  
  \begin{scope}[shift={(8.6,0)}]
  \draw[help lines] (-2.2,-1.2) grid (3.4,4.2);
  \draw[black,->,thick] (-2.2,0) -- (3.4,0);
  \draw[black,->,thick] (0,-1.2) -- (0,4.2);

  \draw[ultra thick,red] (-2,0) -- (3,0);
  \draw[thick,blue] (-2.1,4.2) -- (0,0);
  \draw[thick,blue] (0,0) -- (1.4,4.2);
  \end{scope}
  
  \end{tikzpicture}

  \caption{The pointwise sum of support functions (blue) translates to the Minkowski sum of the associated Newton polytopes (red intervals). }
  \label{fig:sum-support-functions}
\end{figure}

\subsection*{Neural Networks.}

For an introduction to neural networks, we refer to \citet{Goodfellow-et-al-2016}.
In this paper, we focus on so-called \emph{rank-$k$ maxout neural networks} for some natural number $k\geq2$. 
They generalize the well-known neural networks with \emph{rectified linear unit} (\emph{ReLU}) activations.
Such a neural network is  based on a directed, acyclic graph with node set $V$ and arcs $A \subseteq V \times V$, where the nodes are called \emph{neurons} and act as computational units.
The $d\geq1$ neurons with in-degree zero are called \emph{input neurons}, all other $s\geq1$ neurons are \emph{maxout units}. We assume that among the maxout units there exists exactly one neuron with out-degree zero, the \emph{output neuron}.
All neurons that are neither input nor output neurons are called \emph{hidden neurons}.
The \emph{size} of the neural network is defined as the number $s$ of maxout units.
We define our neural networks without biases, see \Cref{rem:biases} below.

Each neuron $v$ in the neural network defines a function $z_v\colon\R^d\to \R$ as follows.
Each of the $d$ input neurons is associated with exactly one element $x_i$ of the input vector $x\in\R^d$. Such an input neuron $v$ simply outputs $z_v(x)=x_i$.
Each maxout unit $v$ comes with a tuple of \emph{weights} $w_{uv}^i\in\R$ for $i=1,\dots,k$ and all $u\in \dinv$, where $\dinv$ is the set of in-neighbors of $v$.
The maxout neuron $v$ represents the following expression dependent on the outputs of its in-neighbors:
\[
z_v=\max\Big\{\sum_{u\in\dinv} w_{uv}^i z_u \Bigm\vert i=1,\dots, k\Big\}\ ,
\]
where $k$ is called the \emph{rank} of the maxout unit.

Finally, the output of the network is defined to be the output $z_v$ for the output neuron~$v$.
Note that our neural network represents a scalar-valued function as we assumed that there is only a single output neuron. 
A maxout network is called \emph{monotone} if all weights of all maxout neurons are nonnegative \citep{mikulincer2022size}. It is called an \emph{input-convex} neural network (ICNN) if negative weights are only allowed on arcs leaving an input neuron \citep{amos2017input}.

For a polytope $P$, we define the \emph{neural network complexity} $\nnc(P)$ as the minimum size of a \emph{rank-$2$} maxout network representing $f_P$. Moreover, we define the \emph{monotone neural network complexity} $\mnnc(P)$ as the minimum size of a rank-$2$ maxout ICNN representing $f_P$.
We will see that, if $f_P$ is a \emph{monotone function}, this exactly equals the minimum size of a monotone rank-$2$ maxout network.
While the previous use of the word `monotone' was for neural networks, recall that a multivariate function is called monotone if, for $x,y\in \R^d$, the componentwise inequality $x\leq y$ implies $f(x)\leq f(y)$.

\begin{remark}[biases]\label{rem:biases}
	In practice, the units usually also involve biases.
	In the following, we will mainly focus on support functions of polytopes, which are positively homogeneous.
	Therefore, we can omit the bias in the definition of neural networks without loss of generality, compare \citet[Proposition~2.3]{hertrich2023towards}.
	With our definition of a maxout network, a function $f\colon\R^d\to\R$ can be represented by such a network if and only if it is CPWL and positively homogeneous.
\end{remark}

\begin{remark}[maxout ranks and ReLU]
	We sometimes focus on rank-2 maxout networks. To justify this,
	note that a rank-$k$ maxout unit can be simulated with $k-1$ \mbox{rank-2} maxout units, which allows to transfer size bounds to other ranks if $k$ is viewed as a fixed constant.
	Moreover, observe that a rank-2 maxout network can be simulated by three parallel ReLU units because $\max\{x,y\}=\max\{0,x-y\}+\max\{0,y\}-\max\{0,-y\}$. Note, however, that this introduces negative weights to the ReLU network even if the maxout network was monotone. Therefore, unlike in the non-monotone case, monotone ReLU networks are indeed weaker than monotone maxout networks.
\end{remark}

\section{Lower Bounds for Monotone and Input-Convex Neural Networks}\label{sec:monotone}

The goal of this section is to prove the following theorem and use it to derive exponential lower bounds on monotone and input-convex neural networks.

\begin{theorem}\label{thm:mnncbound}
	Let $P$ be a polytope. Then $\xc(P)\leq 2\cdot\mnnc(P)$.
\end{theorem}

Before we dive into proving \Cref{thm:mnncbound}, however, we argue that for computing a monotone function, monotone and input-convex maxout networks are equally efficient. 

\begin{proposition}\label{prop:inputconvex}
	If a maxout ICNN represents a monotone function $f$, then it can be converted into a monotone neural network of the same size that represents the same function.
\end{proposition}
\begin{proof}
	Let $\mathcal{N}_0$ be the maxout ICNN of size $s$ representing $f$. Without loss of generality we assume that every input neuron has a connection to every maxout unit. We can always achieve this by introducing arcs with weights equal to zero. This does not increase the number of neurons and therefore not the size of the neural network according to our definition of \emph{size}.\footnote{However, it does potentially increase the number of arcs and therefore (after modifying the weights) the number of parameters of the network.} Let $v_1,v_2,\dots,v_s$ be a topological order of the maxout units of $\mathcal{N}_0$.
	
	We construct a sequence of equivalent neural networks $\mathcal{N}_1,\dots,\mathcal{N}_{s-1}$, removing incoming negative weights neuron by neuron. We only change weights on arcs leaving input neurons.
	All other weights remain unchanged in this process, as they are already nonnegative, because $\mathcal{N}_0$ is input-convex. Our construction ensures that in $\mathcal{N}_p$ none of the neurons $v_1$ to $v_p$ will have negative weights on their incoming arcs.
	
	To go from $\mathcal{N}_{p-1}$ to $\mathcal{N}_p$ for some $p \geq 1$, we leave the first $p-1$ maxout neurons unchanged.
	Let $I\subseteq V$ be the set of input neurons. Consider the weights $w_{uv}^i$ on arcs entering $v\coloneqq v_p$ from an input neuron $u\in I$. Let $\gamma_u\coloneqq\min_{i=1,\dots,k} w_{uv}^i$ be the smallest such weight for each $u\in I$, which might be negative or positive. Observe that
	{\small
		\begin{align}\label{eq:newweights2}
			z_v&=\max\Big\{\sum_{u\in I} (w_{uv}^i - \gamma_u ) z_u + \sum_{u\in\dinv\setminus I} w_{uv}^i z_u \Bigm\vert i=1,\dots,k\Big\} + \sum_{u\in I}  \gamma_u  z_u  \enspace .
		\end{align}
	}%
	This allows to construct $\mathcal{N}_p$ from $\mathcal{N}_{p-1}$ by the following operation. 
	Each weight $w_{uv}^i$ of $v=v_p$ coming from an input neuron $u\in I$ is set to $w_{uv}^i - \gamma_u\geq0$.
	To make up for this change, one has to correct the weights from input neurons to neurons $v_q$ with $q>p$ to incorporate the term outside the maximum in \eqref{eq:newweights2}.
	More precisely, if $\widetilde{v}\coloneqq v_q$ is an out-neighbor of $v=v_p$ and $u\in I$ is an input neuron, then we need to update the weight $w_{u\widetilde{v}}^i$ to $w_{u\widetilde{v}}^i+w_{v\widetilde{v}}^i\gamma_u$. It is easy to verify that this weight update exactly recovers the missing term from \eqref{eq:newweights2} when computing $z_{\widetilde{v}}$.
	This might introduce new negative weights, but all of them are associated with arcs from input neurons to neurons in the topological order after~$p$.
	So, after doing this operation, we obtain an equivalent neural network $\mathcal{N}_p$ with only nonnegative weights associated with the neurons $v_1$ to $v_p$.
	
	We continue that way until we obtain $\mathcal{N}_{s-1}$, where the only weights that could potentially be negative are those between an input neuron and the output neuron $v_s$.
	We show that, in fact, these weights are nonnegative, too, and $\mathcal{N}_{s-1}$ is our desired monotone neural network computing $f$.
	
	To this end, we first show by induction on $p=1,\dots,s-1$ that for a negative unit vector $x=-e_j$ as input, each neuron $v=v_p$ outputs $z_v(-e_j)=0$. If $p=1$, then the only incoming edges of $v=v_1$ are coming from input neurons. Moreover, since we are inputting $-e_j$, only the $j$-th input neuron, call it $u$, propagates a nonzero value. By the choice of $\gamma_u$, there must be an index $i$ such that our updated weight $w_{uv}^i - \gamma_u$ is exactly zero, while it is nonnegative for all other indices. Therefore, upon input $-e_j$, the maximum expression evaluates to $0$, settling the induction start. For the induction step, observe that the very same argument applies, since by induction all previous neurons output zero upon input $-e_j$.
	
	As a consequence of the latter claim,
	we obtain that the output $z_{v_s}(-e_j)$ of $\mathcal{N}_{s-1}$ is precisely the negated weight from $u$ to $v_s$. Since $z_{v_s}(0)=0$, monotonicity of the final function implies that the output on $-e_j$ must be nonpositive, implying that the weight must be nonnegative, as claimed. Thus, we just showed that also the incoming weights of $v_s$ must be nonnegative in $\mathcal{N}_{s-1}$, implying that $\mathcal{N}_{s-1}$ is monotone.
\end{proof}

Observe that the support function $f_P$ of a polytope $P$ is monotone if and only if $P$ is contained in the nonnegative orthant $\R_{\geq0}^d$. \Cref{prop:inputconvex} implies that for the definition of $\mnnc(P)$ of a polytope $P\subseteq\R_{\geq0}^d$, it does not matter whether we use monotone or input-convex neural networks.

We now prove \Cref{thm:mnncbound} in two steps, represented by the next two propositions. The first proposition shows that small ICNNs imply small extended formulations of the epigraph. For a convex CPWL function, the \emph{epigraph} is the polyhedron $\epi(f)=\{(x,t)\in\R^d\times\R\mid t\geq f(x)\}$.
The proposition is related to earlier works connecting the extension complexity to Boolean and arithmetic circuit complexity, see e.g., \citet{fiorini2021strengthening,hrubevs2023shadows}, and also the result by \citet{Goemans:2015} on using sorting networks for constructing an extended formulation of the permutahedron. A similar statement also appears in the paper introducing ICNNs, where \citet{amos2017input} prove that ICNN inference can be written as a linear program.

\begin{proposition}\label{prop:xcepi}
	If $f\colon\R^d\to\R$ is represented by a rank-$k$ maxout ICNN of size $s\geq1$, then $\xc(\epi(f))\leq ks$.
\end{proposition}
\begin{proof}
	Based on the ICNN, we construct an extended formulation for $\epi(f)$ of size $ks$.
	For each neuron $v$, we introduce a variable $y_v$.
	For ease of notation, we identify $y_v$ for an input neuron $v$ with the corresponding entry of the input vector $x$. Similarly, if $v$ is the output neuron, we identify $y_v$ with $t$. Then, the extended formulation for the epigraph is given by the system
	
	\begin{align*}
		y_v \geq \sum_{u\in\dinv} w_{uv}^i y_u \quad\forall\ i=1,\dots,k\quad\forall\text{ maxout units } v\enspace .
	\end{align*}
	
	We claim that this formulation consisting of $ks$ inequalities exactly describes $\epi(f)$.
	The crucial ingredient is the equivalence
	\begin{align*}
		y_v \geq \sum_{u\in\dinv} w_{uv}^i y_u \quad\forall\ i=1,\dots,k \quad\Leftrightarrow\quad y_v \geq \max_{i=1,\dots,k} \sum_{u\in\dinv} w_{uv}^i y_u \enspace .
	\end{align*}

	Indeed, if $(x,t)\in\epi(f)$, setting $y_v\coloneqq z_v(x)$ to the value represented by each hidden neuron $v$ yields by construction a feasible solution with respect to all inequalities.
	Conversely, given a feasible solution $(x,t,y)$, one can inductively propagate the validity of the inequality $y_v\geq z_v(x)$ along a topological order of the neurons.
        Note that, here in the induction step, we need to use that weights between hidden neurons are nonnegative, as the network is input-convex.
        The potentially negative weights on outgoing edges of the input neurons are no problem because no other variable is multiplied by them in the propagation. 
        Then, for the output neuron, we obtain $t\geq f(x)$ and thus $(x,t)\in\epi(f)$.
\end{proof}

\begin{remark}
	In the literature, sometimes maxout networks with different maxout ranks at different units are considered. Suppose each maxout unit $v$ has its own rank $k_v\geq 2$. It is not difficult to see that the proof above generalizes to obtain a bound of $\sum_{v} k_v$ on the extension complexity of the epigraph.
\end{remark}

The second ingredient for proving \Cref{thm:mnncbound} is that the extension complexity of a polytope equals the extension complexity of the epigraph of its support function.
This is related to other previous results about the extension complexity being (almost) preserved under various notions of duality, see, e.g., \citet{martin1991using,gouveia2013nonnegative,weltge2015sizes}.

\begin{proposition}\label{prop:xcdual}
	It holds that $\xc(P) = \xc(\epi(f_P))$ for all polytopes $P$ with $\dim(P)\geq1$.
\end{proposition}
\begin{proof}
	The key idea behind this statement is the well-known fact that the dual of the epigraph of $f_P$ as a cone in $\R^{d+1}$ equals the homogenization of $P$, see, e.g., \citet{Bertsekas:2009,Rockafellar1970}. 
	We give a short geometric proof of this fact to show the ingredients of the construction. 
	A point $(c,t)$ in the epigraph $\epi(f_P)$ fulfills the equivalent conditions
	\begin{align*}
		t \geq f_P(c) \ \Leftrightarrow\ t \geq \max_{x \in P} c^{\top}x\ \Leftrightarrow\ t \geq c^{\top}x \;\forall x \in P\ \Leftrightarrow\ c^{\top}x - t \leq 0\;\forall x \in P \ .
	\end{align*}
	This yields the equivalent representation 
	\begin{align*}
		\epi(f_P) = \Big\{(c,t) \in \R^d \times \R \Bigm\vert (c^{\top},t)\begin{pmatrix}x\\-1\end{pmatrix} \leq 0\; \forall x \in P \Big\} \enspace .
	\end{align*}
	This is the polar cone of the cone over $P \times \{-1\}$ that is
	\begin{align*}
		\epi(f_P) = \text{cone}\left(\left\{ (x,-1)\mid x \in P \right\}\right)^\star \enspace .
	\end{align*}
	To finish the proof, we argue that (de-)homogenization and cone polarity preserve extension complexity.
	The claim about (de-)homogenization follows from the discussion before \citet[Theorem~6]{gouveia2013nonnegative}.
	The statement about cone polarity is implied by the discussion before \citet[Proposition~2]{gouveia2013nonnegative}.   
\end{proof}

Now we are ready to prove \Cref{thm:mnncbound}.

\begin{proofwithcaption}{Proof of \Cref{thm:mnncbound}}
	Combine \Cref{prop:xcepi,prop:xcdual}. The factor $2$ arises from \Cref{prop:xcepi} as $\mnnc(P)$ is defined with $k=2$.
\end{proofwithcaption}

Now we apply \Cref{thm:mnncbound} to obtain strong lower bounds on monotone and input-convex neural networks based on known bounds on the extension complexity.
For a complete graph $(V,E)$ with $n$ vertices and a weight vector $c=(c_e)_{e\in E} \in \R^E$, let $f_{\mathrm{MAT}}(c)$ be the value of a maximum weight matching and $f_{\mathrm{TSP}}(c)$ the value of a longest\footnote{This is to make the function convex; this is equivalent to finding the shortest tour with respect to the negated weights.} traveling salesperson (TSP) tour with respect to $c$.
We denote the matching polytope by $P_{\mathrm{MAT}}$ and the TSP polytope by $P_{\mathrm{TSP}}$. 

\begin{theorem} If a monotone or input-convex maxout network represents $f_{\mathrm{MAT}}$ or $f_{\mathrm{TSP}}$, then it must have size at least $2^{\Omega(n)}$.
\end{theorem}
\begin{proof}
	The functions $f_{\mathrm{MAT}}$ and $f_{\mathrm{TSP}}$ are the support functions of $P_{\mathrm{MAT}}$ and $P_{\mathrm{TSP}}$, respectively.
	Both have extension complexity $2^{\Omega(n)}$ \citep{rothvoss2017matching}.
	Now, the result follows by \Cref{thm:mnncbound}.
\end{proof}

In the same way, one can prove lower bounds for neural networks computing the support function of any polytope with high extension complexity, e.g., for neural networks solving the MAX-CUT problem or the stable set problem.
The corresponding lower bounds on the extension complexity were first proven by \citet{fiorini2015exponential}, who also derived that the extension complexity of the TSP polytope is at least $2^{\Omega(\sqrt{n})}$, before \citet{rothvoss2017matching} improved it to $2^{\Omega(n)}$.

We would like to explicitly highlight one additional result building on the following, which we will use again later. 

\begin{theorem}[{\citet{rothvoss2013some}}] \label{thm:xc-matroid}
	There is a family of matroids $M_n$ on $n$ elements whose matroid base polytopes have extension complexity $2^{\Omega(n)}$.  
\end{theorem}

In fact, the proof of this theorem in \citet{rothvoss2013some} shows that almost all matroid base polytopes must have exponential extension complexity.
With \Cref{thm:mnncbound} and \Cref{prop:inputconvex}, we get the following implication. For a family of matroids $M_n$ on $n$ elements and a weight vector $c\in\R^n$, let $f_n(c)$ be the value of the maximum weight basis in~$M_n$.

\begin{theorem}
	There exists a family of matroids $M_n$ on $n$ elements such that every monotone or input-convex maxout neural network representing $f_n$ must have size $2^{\Omega(n)}$.
\end{theorem}

For machine learning applications, it is arguably less important to obtain neural networks that \emph{exactly} represent a given function.
Instead, approximating the desired output is often sufficient.
We demonstrate that also in the approximate setting, lower bounds on the extension complexity can be transferred to neural networks. We would like to emphasize that there exist many different ways to define what it means to approximate a given function or problem ``sufficiently well'', both in optimization and machine learning.
The ``correct'' notion always depends on the context.
See also \citet{misiakos2022neural} for a discussion of how approximations can be translated between polytopes and neural networks.

\begin{theorem}\label{thm:matchinginapprox}
	Suppose a monotone or input-convex maxout neural network represents a function $f$ such that $f_{\mathrm{MAT}}(c)\leq f(c) \leq (1+\epsilon) f_{\mathrm{MAT}}(c)$ for all $c\in\R_{\geq0}^E$. Then the neural network must have size at least $2^{\Omega(\min\{n,1/\epsilon\})}$.
\end{theorem}
\begin{proof}
	Consider the Newton polytope $P_f$ of the function $f$ computed by the neural network.
	We slightly modify $P_f$ to obtain a polytope $Q$ in the following sense: firstly, we want to include only nonnegative vectors of $P_f$, and secondly, we also want to include all nonnegative vectors that are smaller than a vector in $P_f$. This can be formulated as
	\begin{align}Q\coloneqq \{x\in\R^E\mid \exists y\in P_f\colon 0\leq x\leq y \}\ .\label{eq:monotonepolytope}\end{align}
	With that modification, $Q$ is a ``monotone polytope''\footnote{Note that this notion of a monotone polytope does not correspond to the notion of monotone neural networks.} as defined in \citet{rothvoss2017matching}.
	
	Let us analyze what we can say about the support function $f_Q$ of the modified polytope for $c\in\R_{\geq0}^E$. 
	Firstly, nonnegativity of $c$ implies $f_Q(c) \leq \max\{c^\top x \mid x\in P_f \cap \R_{\geq0}^E\} \leq f(c)\leq (1+\epsilon) f_{\mathrm{MAT}}(c)$. Secondly, we claim that also $f_Q(c)\geq f_{\mathrm{MAT}}(c)$ holds for all $c\in\R_{\geq0}^E$. To see this, assume the contrary, namely the existence of some $c\in\R_{\geq0}^E$ with $f_Q(c)< f_{\mathrm{MAT}}(c)$. Let $x\in P_{\mathrm{MAT}}$ be an optimal vertex of the matching polytope in $c$-direction. It follows that $x\notin Q$. As $P_{\mathrm{MAT}}\subseteq\R_{\geq0}^E$, by definition of $Q$, it follows that $P_f$ and the set $X^+\coloneqq x+\R_{\geq0}^E$ are disjoint. As these two sets are polyhedra, there must exist a separating hyperplane, that is, some vector $a\in \R^E$ and some value $\gamma\in\R$ with $a^\top y \leq \gamma$ for all $y\in P_f$ and $a^\top y > \gamma$ for all $y\in X^+$. The latter implies that $a$ is nonnegative, as the recession cone of $X^+$ is the nonnegative orthant. However, we then have $f_{\mathrm{MAT}}(a)\geq a^\top x > \gamma \geq f(a)$, contradicting the assumption of the theorem.	
	
	In conclusion, also for $Q$ we have the inequality $f_{\mathrm{MAT}}(c)\leq f_{Q}(c) \leq (1+\epsilon) f_{\mathrm{MAT}}(c)$ for all $c\in\R_{\geq0}^E$.
	As argued around \citet[Equation~7]{rothvoss2017matching}, this is equivalent to $P_{\mathrm{MAT}}\subseteq Q \subseteq (1+\epsilon)P_{\mathrm{MAT}}$.
	By \citet[Corollary~4.2]{rothvoss2017matching}, this implies that $\xc(Q)\geq 2^{\Omega(\min\{n,1/\epsilon\})}$; see also \citet{sinha2018lower,braun2014matching}.
	The definition \eqref{eq:monotonepolytope} of $Q$ implies $\xc(Q)\leq \xc(P_f)+2\cdot\lvert E\rvert\leq \xc(P_f)+\mathcal{O}(n^2)$. Thus, we obtain the same asymptotic lower bound on $\xc(P_f)$.
	The statement then follows by \Cref{thm:mnncbound}.
\end{proof}

In particular, monotone and input-convex neural networks cannot serve as \emph{fully-poly\-nomial approximation schemes} for the matching problem, in the sense that they cannot have polynomial size in both $n$ and $1/\epsilon$ to approximate the matching problem to $\epsilon$-precision.
We understand \Cref{thm:matchinginapprox} as a prototype for the fact that, in principle, inapproximability results from extended formulations can be transferred to neural networks. See, e.g., a framework for such lower bounds in \citet{braun2015approximation} and inapproximability up to a factor $2$ for vertex cover in \citet{bazzi2019no}.

\section{Virtual Extension Complexity}

In this section, we study the novel concept of virtual extension complexity $\vxc(P)$. We first prove that $\vxc(P)$ lower-bounds the size of general neural networks, motivating the study of $\vxc(P)$ from a machine learning perspective. We then argue that, even though virtual extended formulations are more general than the ordinary extended formulations, they can still be used to optimize efficiently via linear programming, motivating the study of $\vxc(P)$ from a combinatorial optimization perspective. Finally, we give an example demonstrating that Minkowski sum can indeed drastically reduce the extension complexity. More precisely, there exist polytopes $P+Q=R$ with $\xc(R)$ being way smaller than $\xc(P)$. This means that for obtaining useful lower bounds on $\vxc(P)$, we indeed need to look at both, $\xc(Q)$ and $\xc(R)$, and cannot just focus on $\xc(R)$.

\subsection{Neural Networks are Virtual Extended Formulations}
\label{sec:nntovirtual}

We prove the following theorem stating that neural network sizes can be lower-bounded through virtual extension complexity.

\begin{theorem}\label{thm:nncbound}
	Let $P$ be a polytope. Then $\vxc(P)\leq 4\cdot \nnc(P)$.
\end{theorem}

The proof of this theorem requires three steps, two of which are already available through \Cref{prop:xcepi,prop:xcdual}. In addition, we need another proposition showing that every maxout network can be written as the difference of two monotone maxout networks each of which has the same size as the original network.
This is related to writing functions represented by neural networks as tropical rational functions \citep{zhang2018tropical,brandenburg2024real}.

\begin{proposition}\label{prop:split}
	If $f\colon\R^d\to\R$ is represented by a rank-$k$ maxout network of size~$s$, then it can be written as a difference $f=g-h$ of two functions $g$ and $h$ that are both representable with a \emph{monotone} rank-$k$ maxout network of size $s$.
\end{proposition}

\begin{proof}
	The proof is similar in spirit to \Cref{prop:inputconvex}, with some crucial differences making sure that we also eliminate negative weights between hidden neurons. In the end, this leads to two instead of one monotone network.
	
	Let $\mathcal{N}_0$ be the maxout network of size $s$ representing $f$. Without loss of generality, we assume that the underlying graph of $\mathcal{N}_0$ is transitively closed. We can always achieve this by introducing arcs with weights equal to zero. This does not increase the number of neurons and therefore not the size of the neural network according to our definition of \emph{size}\footnote{As above, this might however increase the number of arcs and parameters of the network.}. Let $v_1,v_2,\dots,v_s$ be a topological order of the maxout units of $\mathcal{N}_0$.
	
	We construct a sequence of equivalent neural networks $\mathcal{N}_1,\dots,\mathcal{N}_{s-1}$, pushing the negative weights neuron by neuron towards the output.
	More precisely, in $\mathcal{N}_p$ none of the neurons $v_1$ to $v_p$ will have negative weights on their incoming arcs.
	
	To go from $\mathcal{N}_{p-1}$ to $\mathcal{N}_p$ for some $p \geq 1$, we leave the first $p-1$ neurons unchanged. Now consider the weights $w_{uv}^i$ of $v\coloneqq v_p$.
	We split them into a positive and negative part, that is, $w_{uv}^i=a_{uv}^i-b_{uv}^i$ with $a_{uv}^i=\max\{w_{uv}^i,0\}$, $b_{uv}^i=\max\{-w_{uv}^i,0\}$.
	Observe that
	{\small
		\begin{align}\label{eq:newweights}
			z_v&=\max\Big\{\sum_{u\in\dinv} \big(a_{uv}^i + \sum_{j\neq i} b_{uv}^j\big) z_u \Bigm\vert i=1,\dots,k\Big\} - \sum_{u\in\dinv}  \sum_{j=1}^k b_{uv}^j  z_u  \enspace .
		\end{align}
	}%
	This allows to construct $\mathcal{N}_p$ from $\mathcal{N}_{p-1}$ by the following operation. 
	Each weight $w_{uv}^i$ of $v=v_p$ is set to $a_{uv}^i + \sum_{j\neq i} b_{uv}^j$.
	To make up for this change, one has to correct the weights of neurons $v_q$ with $q>p$ to incorporate the term outside the maximum in \eqref{eq:newweights}.
	More precisely, if $\widetilde{v}\coloneqq v_q$ is an out-neighbor and $u\in\dinv$ is an in-neighbor of $v=v_p$, then we need to update the weight $w_{u\widetilde{v}}^i$ to $w_{u\widetilde{v}}^i-w_{v\widetilde{v}}^i\sum_{j=1}^kb_{uv}^j$. It is easy to verify that this weight update exactly recovers the missing term from \eqref{eq:newweights} when computing $z_{\widetilde{v}}$.
	This might introduce new negative weights, but all of them are associated with neurons in the topological order after~$p$.
	So, after doing this operation, we obtain an equivalent neural network $\mathcal{N}_p$ with only nonnegative weights associated with the neurons $v_1$ to $v_p$.
	
	We continue that way until we obtain $\mathcal{N}_{s-1}$, where all weights except those of the output neuron are nonnegative.
	For the output neuron, we perform the same splitting operation of the weights into positive and negative parts as for all previous neurons, obtaining  \eqref{eq:newweights} for  $v\coloneqq v_s$. This allows us to define $g$ and $h$ via the right-hand side of \eqref{eq:newweights}: $g$ is the maximum expression and $h$ is the part behind the minus sign. By definition, this implies $f=z_{v_s} = g-h$.
	Now observe that by our construction, $g$ is already a maxout expression with only nonnegative weights. Furthermore, observe that $h$ can be artificially converted into such a maxout expression by taking the maximum over the same linear expression in the $z_u$-terms $k$ times.
	
	Hence we can obtain two monotone neural networks computing $g$ and $h$ by simply copying all neurons except the output neuron from $\mathcal{N}_{s-1}$ and using the maxout expressions described above as weights of the respective output neurons in the two neural networks.
\end{proof}

\begin{proofwithcaption}{Proof of \Cref{thm:nncbound}.}
	The statement simply follows from combining \Cref{prop:split,prop:xcepi,prop:xcdual}. The factor $4$ arises from one factor $2$ through \Cref{prop:split} and another factor $2$ from \Cref{prop:xcepi} as $\nnc(P)$ is defined with $k=2$.
\end{proofwithcaption}

\subsection{Optimizing over Virtual Extended Formulations}\label{sec:optimize}

One of the main reasons to study extension complexity is that it quantifies how well a given problem can be formulated as a linear program. 
Once we have a small-size extended formulation for a polytope $P$, we can efficiently optimize over $P$ by solving a single linear program.
In this section, we argue that virtual extension complexity allows a natural progression of this idea: it describes the power of differences of two linear programs.
Once we have a small-size virtual extended formulation for a polytope $P$, we can optimize efficiently over $P$ by solving two linear programs.

To this end, assume that we have polytopes $P+Q=R$ with small extended formulations for $Q$ and $R$. Switching to support functions implies that for all objective directions $c\in\R^d$ we obtain $\max_{x\in P} c^\top x +\max_{x\in Q} c^\top x =\max_{x\in R} c^\top x$; see also \Cref{fig:correspondence-polytope-support-function,fig:sum-support-functions}.
Thus, to obtain the optimal objective \emph{value}, it is sufficient to have small extended formulations for $Q$ and $R$, as we can really just optimize over $Q$ and $R$ and subtract the results. We now argue that this is basically also true for the \emph{solution} itself.

\begin{proposition}\label{prop:unique}
	Let $P+Q=R$ be polytopes and let $c\in\R^d$ be an objective direction such that the linear program $\max_{x\in R} c^\top x$ has a unique optimal solution~$x^R$. Then also the linear programs $\max_{x\in P} c^\top x$ and $\max_{x\in Q} c^\top x$ have unique solutions~$x^P$ and~$x^Q$, respectively, and it holds that $x^P=x^R-x^Q$.
\end{proposition}

The proof is based on the following fact about faces of Minkowski sums, which can be found, e.g., in \citet{grunbaum2003convex}.

\begin{lemma}\label{lem:MinkFaces}
	Let $P+Q=R$ be polytopes and $c$ be an objective direction. Then, the optimal face in $c$-direction of $R$ is the Minkowski sum of the optimal faces in $c$-direction of $P$ and $Q$, respectively.
\end{lemma}

\begin{proofwithcaption}{Proof of \Cref{prop:unique}.}
	As for $R$ the optimal face in $c$-direction is the singleton $\{x^R\}$, by \Cref{lem:MinkFaces} the corresponding faces of $P$ and $Q$ must be singletons, too, namely $\{x^P\}$ and $\{x^Q\}$ with $x^P+x^Q=x^R$.
\end{proofwithcaption}

\Cref{prop:unique} implies that, for \emph{generic} objective directions $c$, we can efficiently find the optimal \emph{solution} within~$P$ by linear programming using small extended formulations for $Q$ and $R$.
However, the situation becomes more tricky if multiple optimal solutions exist within $R$, and then potentially also within $Q$.
Indeed, in this case we would not know which of the optimal solutions we should subtract in order to obtain a solution that is feasible for $P$.
Therefore, we need to be more careful to handle non-generic objective functions $c$. One solution would of course be to randomly perturb the objective function $c$ by a tiny amount such that the optimal solution in $R$ becomes unique.
In practice, however, it is not clear how to ensure that the perturbation is actually small enough and generic enough at the same time.

The key idea is that we need to make sure that ties between equally good solutions are broken consistently across $Q$ and $R$. 
Therefore, we suggest a deterministic procedure based on a \emph{lexicographic objective function}, which is a standard notion in multi-criteria optimization.
By this, we mean a sequence of objective functions where later entries are only considered to break ties.

\begin{proposition} \label{prop:lexicographic-face}
	Let $P+Q=R$ be polytopes and $(c_1^\top x, \dots, c_d^\top x)$ be a lexicographic objective function.
	If $C\coloneqq\{c_1, \dots, c_d\}$ forms a basis of $\R^d$, the sets of maximizers for $P$, $Q$ and $R$ are singletons $\{x^P\}, \{x^Q\}$ and $\{x^R\}$, respectively, and $x^P + x^Q = x^R$. 
\end{proposition}
\begin{proof}
	For $i=1,\dots,d$, let $F^P_i$ be the set of optimal solutions over $P$ with respect to the lexicographic objective function $(c_1^\top x, \dots, c_i^\top x)$.
	Let $F^Q_i$ and $F^R_i$ be defined analogously.
	We show by induction on $i$ that $F^P_i+F^Q_i=F^R_i$.
	For $i=1$, this follows from \Cref{lem:MinkFaces}.
	For $i>1$, observe that $F^P_i$ is the face in $c_i$-direction of $F^P_{i-1}$, and analogously for $Q$ and $R$.
	Hence, we can apply \Cref{lem:MinkFaces} again to complete the induction step.
	As a consequence, we obtain that $F^P_d+F^Q_d=F^R_d$.
	Moreover, since $C$ is a basis of $\R^d$, it follows that any two optimal solutions with respect to the lexicographic objective function must be equal, implying that $F^P_d$, $F^Q_d$, and $F^R_d$ are singletons.
\end{proof}

This lays the basis for \Cref{alg:optvirtextform}.

\begin{algorithm2e}[htb]
	\caption{Solve $\max_{x\in P} c^\top x$ where $P$ is given through extended formulations of two polytopes $Q$ and $R$ with $P+Q=R$.}
	\label{alg:optvirtextform}
	let $c_1\coloneqq c\in\R^d$\;
	compute an arbitrary basis $C\coloneqq\{c_1,c_2,\dots,c_d\}$ of $\R^d$\;
	let $x^Q$ and $x^R$ be optimal solutions within $Q$ and $R$, respectively, with respect to the lexicographic objective function $(c_1^\top x,c_2^\top x,\dots,c_d^\top x)$\label{line:LP}\;
	\Return{$x^R-x^Q$}\;
\end{algorithm2e}

\begin{theorem} \label{thm:algo-two-extended-formulations}
	\Cref{alg:optvirtextform} correctly returns an optimal solution of $\max_{x\in P} c^\top x$ and can be implemented with a running time that is polynomial in the encoding sizes of $c$ and the extended formulations for $Q$ and $R$.
\end{theorem}
\begin{proof}
	Let $x^P, x^Q, x^R$ be the unique optimizers as guaranteed by \Cref{prop:lexicographic-face}. 
	As they fulfill $x^P + x^Q = x^R$, the algorithm correctly returns $x^P = x^R - x^Q$.
	Furthermore, since $x^P$ is optimal with respect to the lexicographic objective, it is in particular optimal with respect to the first component of the objective, which is $c^\top x$.
	Therefore, the algorithm returns an optimal solution for $\max_{x\in P} c^\top x$.
	
	For the running time, the only interesting step is line \ref{line:LP}.
	One way to see that this runs in polynomial time is to first solve with respect to objective $c_1$, obtain the optimal objective value $\lambda_1$, then add the constraint $c_1^\top x = \lambda_1$, and solve with respect to objective $c_2$, and so on.
	After solving $d$ linear programs for each of the two polytopes $Q$ and $R$, we obtain the desired optimal solutions with respect to the lexicographic objective.
	Depending on the precise linear programming algorithm used, one can handle the lexicographic objective function directly in a single linear programming computation for each of the two polytopes.
\end{proof}

\subsection{Extension Complexity and Minkowski Sums}\label{sec:genperm}

For a better understanding of virtual extension complexity, the behavior of extension complexity under Minkowski sum is crucial. 
It is well-known that $\xc(R) \leq \xc(P) + \xc(Q)$ for $P + Q = R$.
This follows from plugging the extended formulations for $P$ and $Q$ into the definition of the Minkowski sum.
However, for bounding virtual extension complexity, one would need to bound $\xc(P)$ in terms of $\xc(Q)$ and $\xc(R)$.
In the following, we present a class of examples demonstrating that one cannot simply focus on $\xc(R)$ for this, but also needs to take~$\xc(Q)$ into account.

Recall that the \emph{regular permutahedron} for $n \in \N$ is the polytope
\[
\Pi_n = \conv\{ (\sigma(1),\sigma(2),\ldots,\sigma(n)) \mid \sigma \text{ is a permutation of }\{1,\dots,n\}\}
\]
arising as the convex hull of all permutations of the vector $(1,2,\ldots,n)$.
The regular permutahedron has the property that all edge directions are of the form $e_i - e_j$ for two unit vectors $e_i, e_j$.
More generally, polytopes fulfilling this form the class of \emph{generalized permutahedra}.
A different characterization is in terms of Minkowski sums: a polytope $P$ is a generalized permutahedron if and only if there is a polytope $Q$ with $P + Q = \lambda\cdot \Pi_n$ for some $\lambda\geq0$, see \citet[Proposition~3.2]{PostnikovReinerWilliams:2008}.
It turns out that every \emph{matroid base polytope} is such a generalized permutahedron, and we can even choose $\lambda=1$ in this case \citep{BorovikGelfandWhite:2003,Fujishige:2005}. 

Hence, we can apply \Cref{thm:xc-matroid} to deduce a statement about summands of $\Pi_n$.

\begin{theorem}
	For $n \in \N$, there is a polytope $\Pi_n$ with extension complexity $\mathcal{O}(\text{poly}(n))$, such that a Minkowski summand of $\Pi_n$ has extension complexity~$2^{\Omega(n)}$.
\end{theorem}
\begin{proof}
	By \citet{Goemans:2015}, the extension complexity of the regular permutahedron $\Pi_n$ is of the order $\Theta(n\log(n))$.
	As discussed, each matroid base polytope $P$ is a Minkowski summand of $\Pi_n$. 
	Now, \Cref{thm:xc-matroid} guarantees that there exist such polytopes with extension complexity~$2^{\Omega(n)}$.
\end{proof}

Considering a matroid polytope $P$ with high extension complexity from the latter proof such that $P + Q = \Pi_n$ for some other generalized permutahedron~$Q$, we cannot exclude that $Q$ also has high extension complexity.
Therefore, without a better understanding of $Q$, this does not yield any useful bound on $\vxc(P)$. 

On a final note, we would like to emphasize that in the definition of the virtual extension complexity, it is important that $Q$ and $R$ are polytopes and not potentially unbounded polyhedra. A simple example is given by the choice of $Q=R=\R^d$, which would immediately imply $\vxc(P)=0$ for any polyhedron $P$, if this were allowed. The important algebraic property that we need to make sure that we can recover $P$ from $Q$ and $R$ is the cancellation property: For polytopes $A,B,C$, it is true that $A+C=B+C$ implies $A=B$. The same is not true for potentially unbounded polyhedra.

\section{Conclusions}

In this paper we focused on proving strong lower bounds for monotone and input-convex neural networks based on existing breakthrough results about extension complexity. Furthermore, we paved the way towards transferring these bounds to general neural networks by introducing the notion of virtual extension complexity and relating it to neural networks.

An obvious question for future research is a thorough analysis of how much more powerful general neural networks are compared to monotone or input-convex neural networks. Can we find a function for which there is an exponential gap in the required size?

Closely related, and from the perspective of polyhedral theory and combinatorial optimization, the most intriguing open question is the following.

\begin{question}\label{qu:open}
	Can $\vxc(P)$ be (much) smaller than $\xc(P)$?
\end{question}

We do not see any clear evidence in either direction for the resolution of \Cref{qu:open}. If the answer is no, then this would imply strong lower bounds on neural networks. For that, it would even be sufficient to prove that $\xc(P)$ cannot be exponentially larger than $\vxc(P)$. On the other hand, if the answer to \Cref{qu:open} is yes, it would imply that solving two linear programs and taking the difference is (much) more powerful than just solving one linear program.

In fact, we are not even aware of an example for which we can prove that $\vxc(P)<\xc(P)$. A first step into this direction is the result by \cite{hertrich2025arithmetic}, showing that for $P$ being the base polytope of a regular matroid on $n$ elements, $\vxc(P)\in O(n^3)$, while the best known upper bound on $\xc(P)$ is $O(n^6)$~\citep{aprile2022regular}. However, it is conceivable that the $\xc$ upper bound is suboptimal and there might be no actual separation.

As for the usual extension complexity, we do not expect $\vxc(P)$ to be polynomial for any polytope~$P$ associated with an NP-hard optimization problem. An obvious, but seemingly challenging approach to lower-bound virtual extension complexity is to take any previously successful method for lower-bounding extension complexity and try to adapt it such that it also works in the virtual case. Most such approaches in the literature rely on the theorem by \citet{yannakakis1991expressing} stating that $\xc(P)$ equals the nonnegative rank of the slack matrix. Therefore, a first step would be to establish a similar, even if weaker, characterization of $\vxc(P)$ in terms of the slack matrix. Beyond this, we think the most promising route might be to focus on lower bounds based on counting arguments, similar to~\cite{rothvoss2013some}, with the aim to obtain an analogue of \Cref{thm:xc-matroid} for $\vxc$.

To summarize, we hope that this paper inspires researchers to find out: Are two LPs better than one?

\subsection*{Acknowledgments}We thank Alex Black, Daniel Dadush, Jes\'us A. De Loera, Samuel Fiorini, Neil Olver, L\'aszl\'o V\'egh, Matthias Walter, and Stefan Weltge for insightful discussions around (virtual) extension complexity that have shaped this paper.
	
Part of this work was completed while Christoph Hertrich was affiliated with Université \mbox{Libre} de Bruxelles, Belgium, and received support from the European Union's Horizon Europe research and innovation program under the Marie Skłodowska-Curie grant agreement No 101153187---NeurExCo.

\bibliographystyle{abbrvnat}
\bibliography{ref}

\appendix

\crefalias{section}{appendix} 

\end{document}